\documentclass[12pt, letterpaper, twoside]{amsart}

\title{Permutations over cyclic groups}

\author[Zolt\'an L. Nagy]{Zolt\'an L\'or\'ant Nagy}
\thanks{This research was supported by Hungarian National Scientific
Research Funds (OTKA) grant 81310.}
\usepackage{amsmath}
\usepackage{amssymb}
\usepackage{amsthm}
\usepackage{graphicx}
\usepackage{amscd}

\usepackage[latin1]{inputenc}

 \newtheorem{theorem}{\em Theorem}[section]
 \newtheorem{lemma}[theorem]{\em Lemma}
 
 \newtheorem{prop}[theorem]{\em Proposition}

\newtheorem{remark}[theorem]{\em Remark}
\newtheorem{defn}[theorem]{\em Definition}

 \makeatletter
 \@addtoreset{equation}{section}
 \makeatother

 \makeatletter
 \@addtoreset{table}{section}
 \makeatother

\begin{document}

\maketitle

\begin{abstract}
Generalizing a result in the theory of finite fields we prove that, apart from a couple of exceptions that can be classified, for any elements $a_1,\dots ,a_m$ of the cyclic group of order $m$, there is a permutation $\pi$ such that $1a_{\pi(1)}+ \cdots +ma_{\pi(m)}=0$. 
\end{abstract} 

\section{Introduction}

\noindent The starting point of the present paper is the following result of G\'acs, H\'eger, Nagy and P\'alv\"olgyi.


\begin{theorem}\cite{Gacs}\label{prim} Let $\{ a_1,a_2,\dots ,a_p\}$ be a multiset in the finite field $GF(p)$, $p$ a prime. Then after a suitable permutation of the indices, either $\sum _iia_i=0$, or   
$a_1=a_2=\cdots =a_{p-2}=a$, $a_{p-1}=a+b$, $a_p=a-b$ for field elements $a$ and $b$, $b\ne 0$.
\end{theorem}

A similar result using a slightly different terminology was obtained by Vinatier \cite{Vin} under the extra assumption that $a_1,\dots ,a_p$, when considered as nonnegative integers, satisfy $a_1+\cdots +a_p=p$.
The former result can be extended to arbitrary finite fields in the following sense.


\begin{theorem}\cite{Gacs} \label{primhatvany} Let $\{ a_1,a_2,\dots ,a_q\}$ be a multiset in the finite field $GF(q)$. There are no distinct field elements $b_1,b_2,\dots ,b_q$ such that $\sum _ia_ib_i=0$ if and only if after a suitable permutation of the indices, $a_1=a_2=\cdots =a_{q-2}=a$, $a_{q-1}=a+b$, $a_q=a-b$ for some field elements $a$ and $b$, $b\ne 0$.
\end{theorem} 

This theorem can be reformulated in the language of finite geometry and also have an application about the range of polynomials over finite fields. For more details, see \cite{Gacs}.

Our aim is to find a different kind of generalization, more combinatorial in nature, which refers only to the group structure. First we extend the result to cyclic groups of odd order.


\begin{theorem}\label{odd} Let $\{ a_1,a_2,\dots ,a_m\}$ be a multiset in the Abelian group $\mathbb{Z}_m=\mathbb{Z}/m\mathbb{Z}$, where $m$ is odd. Then after a suitable permutation of the indices, either $\sum _iia_i=0$, or   
$a_1=a_2=\cdots =a_{m-2}=a$, $a_{m-1}=a+b$, $a_m=a-b$ for elements $a$ and $b$, $(b,m)=1$.
\end{theorem} 

\medskip

 The situation is somewhat different if the order of the group is even. In this case we have to deal with two types of exceptional structures. 
The following statements are easy to check.

\begin{prop}\label{except}
Let $m$ be an even number represented as $m=2^kn$, where $n$ is odd.\begin{itemize} \item[(i)]If a multiset $M=\{ a_1,a_2,\dots ,a_m\}$ of $\mathbb{Z}_m$ consists of elements having the same odd residue $c$ mod ${2^k}$, then $M$ has no permutation for which $\sum _iia_i=0$ holds.
\item[(ii)] If $M=\{a,a, \ldots, a+b, a-b\}$ mod ${m}$, where $a$ is even and $(b,m)=1$ holds, then $M$ has no permutation for which $\sum _iia_i=0$ holds. \end{itemize} These two different kind of structures we call \textsl{homogeneous and inhomogeneous exceptional multisets}, respectively.
\end{prop}

\begin{theorem}\label{even} Let $M=\{ a_1,a_2,\dots ,a_m\}$ be a multiset in the Abelian group $\mathbb{Z}_m$, $m$ even.  If $M$ is not an exceptional multiset as defined in Proposition \ref{except}, then after a suitable permutation of the indices $\sum _iia_i=0$ holds.
\end{theorem}

The presented results might be extended in  different directions. One may ask whether there exists a permutation of the elements of a given multiset $M$ of $\mathbb{Z}_m$ (consisting of $m$ elements), for which the sum $\sum _iia_i$ is equal to a prescribed element of $ \mathbb{Z}_m$.
This question is related to a conjecture of  Britnell and Wildon, see \cite[p.~20]{britnell}, which can be reformulated as follows. Given a multiset $M=\{ a_1,a_2,\dots ,a_m\}$ of $\mathbb{Z}_m$,  all  elements of $\mathbb{Z}_m$  are admitted as the value of the sum $\sum_{i=1}^m ia_{\pi(i)}$ for an appropriate permutation $\pi \in S_m$, unless one of the following holds: \begin{itemize} \item $M=\{a,\ldots,a, a+b, a-b\}$, \item there exists a prime divisor $p$ of $m$ such that all elements of $M$ are the same mod $p$.
\end{itemize} 

Our result may in fact be considered as a major step towards the proof of their conjecture, which would provide a classification of values of determinants associated to special types of matrices. When $m$ is a prime, the conjecture is an immediate consequence of Theorem \ref {prim} and Lemma \ref{trafo1} (ii). 
  
As for another direction, these questions are also meaningful for arbitrary finite Abelian groups, but to find the exact characterization appears to be a difficult task in general. For example, in the Klein group $\mathbb{Z}_2^2$, the multiset consisting of all different group elements has no zero `permutational sum', whereas all other multisets do have. Meanwhile in  the group $\mathbb{Z}_2^3$, all multisets have a permutational sum which is zero. 


As it was briefly explained in \cite{Gacs}, the problem has a connection to Snevily's conjecture \cite{Sn}, solved recently by Arsovski \cite{Snevily}.  It would be natural to try to adapt the techniques which were successful for Snevily's problem, but our problems are apparently more difficult. In order to prove Theorems \ref{prim} and \ref {primhatvany}, we had to replace the relatively simple approach of Alon  \cite{Alon1} by a more delicate application of the Combinatorial Nullstellensatz \cite{Alon}, \cite{Karolyi} and we do not see how Theorem \ref{odd}, for example, could be obtained by the method of   \cite{Snev}.


The paper is organized as follows. In Section 2, we collect several simple observations that are used frequently throughout the paper and sketch our  proof strategy. Section~3 is devoted to the proof of Theorem \ref{odd}. In Section 4 we will verify  Theorem \ref{even} in some particular cases, whose proof do not exactly fit in the general framework (and may be skipped at a first reading). The complete proof, which is more or less parallel to that of Theorem \ref{odd}, is carried out in Section 5. 

\bigskip


\section{Preliminaries}

\bigskip


\begin{defn} Let $M = \{ a_1,\dots ,a_m\}$ be a multiset in $\mathbb{Z}_m$. A permutational sum of the elements of $M$ is any sum of the form $\sum_{i=1}^m ia_{\pi(i)}$, $\pi \in S_m$. If, after some rearrangement, we fix the order of the elements of $M$, then the permutational sum of $M$ considered as a sequence $(a_1, \ldots, a_m)$ is simply  $\sum _{i=1}^m ia_i$.
\end{defn}

Accordingly, the aim is to determine which multisets admit a zero permutational sum. This property is invariant under certain transformations.

\begin{lemma} \label{trafo1}
Let $m$ be odd, and $M$ be a multiset in $\mathbb{Z}_m$ of cardinality $m$.

\begin{itemize}

\item[(i)]   If no permutational sum of $M$ admits the value $0$, then the same holds for any translate $M+c$ of $M$, and also for any  dilate $cM$ in case $(c,m)=1 $.

\item[(ii)] If the permutational sums of $M$ admit a value $w$, then they also admit the value $kw$ for every integer $k$ with $(m,k)=1$. As a consequence, if $(m,w)=1$, then the permutational sums take at least $\varphi(m)$ different values.

\item[(iii)] Assume that  $M$ has the exceptional structure, i.e. $M= \{ a,\dots, a, a+b, a-b \}$ where $(b,m)=1$. Then the permutational sums of $M$ admit each element of $\mathbb{Z}_m$ except zero.
\end{itemize}
\end{lemma}

\begin{proof}
Parts (i) and (iii) are straightforward, for $1+2+\ldots +m \equiv 0 \pmod m$. Part (ii) follows from the fact that $\pi\in S_m$ holds for the function $\pi$ defined by $\pi (i)=ki $ .
\end{proof}

The sumset or Minkowski sum $C+D$ of two subsets $C$ and $D$ of an Abelian group $G$ written additively is $C+D=\{c+d ~|~ c\in C, d\in D\}$.
The following statement is folklore. 

\begin{lemma}\label{sumset} For $C, D \subseteq\mathbb{Z}_m$, $|C|+|D|>m$ implies $C+D=\mathbb{Z}_m$. 
\end{lemma} 

In the remaining part of this section, we sketch the proof of Theorem \ref{odd}. Recall that the arithmetic function $\Omega(n)$ represents the total number of prime factors of $n$. Similarly to the classical result in zero-sum combinatorics due to Erd\H os, Ginzburg and Ziv \cite{erdos}, we proceed by induction on $\Omega(m)$.
 The initial case is covered by Theorem \ref{prim}, so in the sequel we assume that $m$ is a composite number and fix a prime divisor $p$ of $m$ and write $m=p^kn$, where $(p,n)=1$.
 
The proof is carried out in several steps (of which the first two will be quite similar to the beginning of the proof of  Theorem \ref{even}).


\bigskip

\noindent {\bf2.1. First step}

\bigskip

\noindent  We introduce the notion of  \textit{initial order} as follows.

\begin{defn}
Let $s=(b_1, b_2, \ldots, b_m) $  be any sequence in $\mathbb{Z}_{m}$.
\begin{itemize}
\item[(i)] A cyclic translate of $s$ is any sequence of the form $(b_i, b_{i+1}, \ldots, b_m, b_1, \ldots, b_{i-1}).$
\item[(ii)] The sequence $s$  is separable (relative to the prime divisor $p$ of $m$) if  equivalent elements mod ${p^l}$ are consecutive for every $1\leq l\leq k$.  
\end{itemize}
\end{defn}

 \noindent Thus separability means that for $1\leq i<j\leq m$ and every $l\leq k$, $a_i \equiv a_j \pmod {p^l}$ implies $a_i \equiv a_h \pmod {p^l}$ for  every $i<h<j$.
Note that one can always order the elements of $M$ into a separable sequence. Choose and fix an initial order of the elements of $M$ such that some cyclic translate of the sequence $(a_1, a_2, \ldots, a_m)$ is separable. 

A useful property of such an ordering is summarized in the following lemma whose proof is straightforward.

\begin{lemma}\label{rem}   Consider a sequence of $m$ elements in $\mathbb{Z}_{m}$, which admits a separable cyclic translate. Partition the elements into $k\geq 3$ consecutive blocks $T_1, \ldots, T_k$. If for an integer $l$, a certain residue $r$ mod ${p^l}$ occurs in every block, then at most two of the blocks may contain an element having a residue different from $r$.
The same conclusion holds if the elements are rearranged inside the individual blocks.
\end{lemma}

Let $(a_1, \ldots a_m)$ be an initial order. Form $p$  consecutive blocks of equal size, denoted by $T_1, T_2, \ldots, T_p$, each containing $m^*:={m/p}$ consecutive elements. More precisely, $$ T_i= \{a_{(i-1)m^*+1}, a_{(i-1)m^*+2}, \ldots, a_{im^*}\}.$$
$S_i$ denotes the sum of the elements in $T_i$, while $R_i$ denotes the permutational sum of the block $T_i$ (as a multiset), that is, $R_i= \sum_{j=1}^{m^*}ja_{j+(i-1)m^*}$.\\
Writing $R=\sum_{i=1}^p R_i$, the permutational sum of $M$ takes the form

$$\Phi= \sum_{j=1}^m ja_j = \sum_{i=1}^p \left(R_i + m^*(i-1)S_{i}\right) = R + m^*\sum_{i=0}^{p-1}iS_{i+1}.$$

\bigskip

\noindent {\bf2.2. Second step}

\bigskip

\noindent Our aim here is to ensure that $m^*\mid \Phi$ holds after a well structured rearrangement of the elements. That is, we want to achieve that $m^*\mid R$ holds.
To this end we allow reordering the elements inside the individual blocks. Such a permutation will be referred to as a \textit{block preserving permutation}.
We distinguish three different cases. 

First, if there is no exceptionally structured block mod ${m^*}$, then by the inductional hypothesis the elements in each block $T_i$ can be rearranged so that $m^*$ divides $R_i$. Thus, after a block preserving permutation, $m^*\mid R$.

Next, if there is an exceptionally structured block $T_i$, then the permutational sums over $T_i$ take $m^*-1$ different values mod ${m^*}$, see Lemma \ref{trafo1} (iii). If there are at least two exceptionally structured blocks, then it follows from Lemma \ref{sumset} that there is a block preserving permutation that ensures  $m^*\mid R$.

Finally, if there is exactly one exceptionally structured block  $\{a,\ldots, a, a+b, a-b\}$ $\pmod {m^*}$, then a permutational sum of this block can take any value except $0$ mod ${m^*}$. 
So after a block preserving permutation we are ready, unless zero is the only value that the other blocks admit, that is, all elements must be the same in each block mod ${m^*}$.\\ This latter case can be avoided by a suitable choice of the initial order in the first step. 
Indeed, translating the initial order cyclically so that it starts with an appropriate element from the exceptional block will break down this structure.

\bigskip

\noindent {\bf2.3. Third step}

\bigskip

\noindent To complete the proof,  based on the relation $m^*\mid \Phi$ we further reorganize the elements to achieve a zero permutational sum, or else to conclude to (one of) the exceptional case(s). We only outline here the  strategy of the proof, as the following section is devoted to the detailed discussion.

As a first approximation, we try to change the order of the blocks to obtain 
 $$ \sum_{i=0}^{p-1}iS_{i+1}\equiv -R':= - \frac{R}{m^*}  \pmod p,$$
which would imply $m \mid \Phi.$ One is tempted to argue that the case $R'\equiv 0 \pmod p$ would be easy to resolve applying  Theorem \ref{prim} for the multiset $\{S_1, \ldots, S_p\}$. As it turns out, the main difficulty is to handle exactly this case, since the  multiset $\{S_1, \ldots, S_p\}$ may have the exceptional structure. A remedy for this is what we call the `braid trick'. The main idea of this tool will be  to consider the transposition of  a pair of elements  whose indices differ by a fixed number $x$    (typically a multiple of $m^*$).  By this kind of transposition of a pair ($a_i, a_{i+x}$), the permutational sum increases by $x(a_i- a_{i+x})$, providing a handy modification.  
 

\bigskip

\section{The case of odd  order} 

\bigskip

\noindent In this section we complete the proof of Theorem \ref{odd}. We continue with the details of the third step outlined in the previous section. 
We distinguish two cases according to whether $R'$ is divisible by $p$ or not.\\

\noindent {\bf3.1. $ R'$ is not divisible by $p$.}
 
\bigskip 
 
\noindent  Note that $ \sum_{i=0}^{p-1}iS_{i+1}$ can be viewed as a permutational sum of the multiset \ $\mathcal{S}=\{S_1, S_2, \ldots, S_p\}$. If there are two elements $S_i \not \equiv S_j \pmod p$, then their transposition changes the value of the permutational sum of $\mathcal{S}$ mod $p$. In particular, the permutational sums of $\mathcal{S}$ admit a nonzero value mod $p$. From Lemma \ref{trafo1} (ii) it follows that they admit each nonzero element of  $\mathbb{Z}_{p}$ and in particular $-R'$ too.

Otherwise, we have $S_1\equiv S_2\equiv \ldots \equiv S_p \pmod p$.
 We use the \texttt{braid trick}: we look at the  pairs $(a_i, a_{i+m^*} )$ for every $i$.
The elements $a_i$ and $a_{i+m^*}$ occupy the same position in two consecutive blocks $T_j, T_{j+1}$. If they have different residues mod ${p}$, then their transposition leaves $R$ intact, hence $R'$ does not change either. On the other hand, $S_j$ and $S_{j+1}$ change whereas each other $S_i$ remains the same, therefore the previous argument can be applied.

Finally, we have to deal with the case when $a_i \equiv a_{i+lm^*} \pmod {p}$ holds for every possible $i$ and $l$. This is the point where we exploit the separability property.  The initial order has changed only inside the blocks during the second step. Since the number of blocks is at least three, it follows from Lemma \ref{rem} that $a_i\equiv a_j \pmod p$ for all $1\leq i<j\leq m$ in $M$. In this case we prove directly that $M$ has a zero permutational sum.
 In view of Lemma \ref{trafo1} (i), we may suppose that every $a_i$ is divisible by $p$.
Consider $M^*:= \{\frac{a_1}{p}, \frac{a_2}{p}, \ldots, \frac{a_m}{p}\}$. Apply the first two steps for this multiset $M^*$. It follows that $M^*$ has a zero permutational sum mod ${m^*}$, which implies that $M$ has a zero permutational sum mod $m$.\\

\noindent{\bf3.2. $ R'$ is divisible by $p$}

\bigskip

\noindent Here our aim is to prove that $p\mid  \sum_{i=0}^{p-1}iS_{i+1}$ holds for a well chosen permutation of the multiset $\mathcal{S}:=\{S_1,\ldots, S_p\}$.
This is exactly the problem what we solved in Theorem \ref{prim}, which implies that we can reorder the blocks (and hence the multiset $M$ itself) as required, except when the multiset $\mathcal{S}$ has the form $\{A, A, \ldots, A, A+B, A-B\}$, with the condition $(B,p)=1$.

Once again, we apply the braid trick.
If $a_i$ and $a_{i+lm^*}$ have different residues mod $p$, then we try to transpose them in order to destroy this exceptional structure. As in Subsection \textsl{$3.1$}, $R$ does not change. We call a pair of elements exchangeable if their indices differ by a multiple of $m^*$.  
  
Thus, a zero permutational sum of $M$ is obtained unless no transposition of two exchangeable elements destroys the exceptional structure of $\mathcal{S}$. The following lemma gives a more detailed description of this situation.

\begin{lemma} \label{speceset1} 
Suppose that no transposition of two exchangeable elements destroys the exceptional structure of $\mathcal{S}$. Then either this exceptional structure can be destroyed by two suitable transpositions, or $M$ contains only three distinct elements : $t, t+B, t-B$ mod $p$ for some $t$ with the following properties:
\begin{itemize}
 
\item $t+B$ occurs only in one block, and only once;
\item $t-B$ occurs only in one block, and only once;
\item $t+B$ and $t-B$ occupy the same position in their respective blocks.
\end{itemize} 
\end{lemma}  

\begin{proof}
Denote by $T^+$ and $T^-$ the blocks for which the sum of the elements is $A+B$ and $A-B$, respectively.
Apart from elements from $T^+$ and $T^-$,  two exchangeable elements must have the same residue mod $p$. Furthermore, if a transposition between $a_j\in T^-$ and $a_{j+lm^*}\notin T^+$  does not change the structure of $\mathcal{S}$, that means $a_j\equiv a_{j+lm^*} \pmod p$ or $a_j\equiv a_{j+lm^*}-B \pmod p$. Similar proposition holds for $T^+$. 

Consider now a set of pairwise exchangeable elements. One of the followings describes their structure:  either they all have the same residue mod $p$, or they have the same residue $t$ mod $p$ except the elements from $T^+$ and $T^-$, for which the residues are $t+B$ and $t-B$, respectively.

Observe that both cases must really occur since the sums $S_i$ of the blocks are not uniformly the same. In particular, there is a full set of $p$ pairwise exchangeable elements having the same residue  mod $p$.

Since the number of blocks is at least $3$, we can apply  Lemma \ref{rem}. We only used block-preserving permutations so far, hence it follows that all elements have the same residue mod $p$ --- let us denote it by $t$ --- except some  $(t+B)$'s in  $T^+$, and the same number of $(t-B)$'s in  $T^-$, in the very same position relative to their blocks.

We claim that this number of different elements in $T^+$ and $T^-$ must be one, otherwise we can destroy the exceptional structure with two transpositions. Indeed, by contradiction, suppose that there exist two distinct set of exchangeable elements where the term corresponding to $T^+$ and $T^-$ is $t+B$ and $t-B$, respectively. Pick a block different from $T^+$ and $T^-$ and denote it by $T$. Then transpose $t+B\in T^+$ and $t \in T$ in the first set, and $t-B\in T^-$ and $t \in T$ in the second set. The new structure of $\mathcal{S'}$ obtained this way is not exceptional any more.
\end{proof}

\begin{lemma}\label{baratunk}
Suppose that $M$ contains only three distinct elements : $t, t+B, t-B$ mod $p$ for some $t$ with the following properties:
\begin{itemize}
 
\item $t+B$ occurs only in one block, and only once;
\item $t-B$ occurs only in one block, and only once;
\item $t+B$ and $t-B$ occupy the same position in their respective blocks.
\end{itemize}
Then either a suitable zero permutational sum exists or the conditions on $M$ hold mod ${p^l}$ for every $l\leq k$, with a suitable $B=B_l$ not divisible by $p$.
\end{lemma}

\begin{proof}
We proceed by induction on $l$. Evidently, it holds for $l=1$.\\
According to  Lemma \ref{trafo1} (i) we may assume that $t \equiv 0 \pmod p$. Let $a^+$ and $a^-$ denote the elements of $T^+$ and $T^-$ for which $a^+ \equiv B \pmod p$ and $a^- \equiv -B \pmod p$. Note that their position is the same in their blocks.

Suppose  that $l\geq 2$ and the conditions  hold mod ${p^{l-1}}$.
Consider the residues of the elements mod ${p^l}$ now. We use again the braid trick. 
Suppose that there exist  $a_i, a_j \not\in \{a^-, a^+\}$ such that $i-j$ is divisible by $p^{k-l}$ but not by $p^{k-l+1}$, and   $a_i \not\equiv a_j \pmod {p^l}$. After we transpose them, (the residue of) $R$ does not change  mod ${p^{k-1}}$, but it changes by $(i-j)(a_j-a_i)\neq 0 \pmod {p^{k}}$. For the new permutational sum thus obtained, $R' \not\equiv 0 \pmod p$ holds, while the multiset $\mathcal{S}$ may change, but certainly  it does not become homogeneous mod $p$. Thus $M$ has a zero permutational sum, as in Subsection 3.1.

Otherwise, in view of Lemma \ref{rem} it is clear that  all the residues must be the same mod ${p^l}$, and we may suppose they are zero, except the residues of $a^+$ and $a^-$. In addition, $a^+ + a^- \equiv 0 \pmod {p^l}$ must hold too, since $R' \equiv 0 \pmod p $. This completes the inductional step.
\end{proof}

Lemma \ref{baratunk} applied for $l=k$ completes the proof of Theorem \ref{odd} when $m=p^k$ is a prime power. In the sequel we assume that   $n\neq 1$. Let $m=p_1^{k_1}p_2^{k_2}\ldots p_n^{k_r}$ be the canonical form of $m$. Note that the whole argument we had so far is valid for any prime divisor $p$ of $m$. Therefore, to complete the proof , we may assume that $M$ has the exceptional structure mod ${p_i^{k_i}}$ as described in Lemma \ref{baratunk} for every $p=p_i$.



\begin{lemma}\label{kivetel} The conclusion of Theorem \ref{odd} holds if $M$ has exceptional structure modulo each $p_i^{k_i}$.
\end{lemma}
 
\begin{proof} 
We look at the permutational sums of $M$ leaving the elements of $M$ in a fixed order $a_1, a_2, \ldots, a_m$ while permuting the coefficients $1, 2, \ldots, m$. 
According to Lemma~\ref{trafo1}~(i) we may assume that all elements,  except two, are divisible by $p_1^{k_1};$   all elements,   except two, are divisible by $p_2^{k_2}$, and so on. It follows that at least $m-2r$ elements are zero mod $m$, so their coefficients are irrelevant. 
So we only have to assign different coefficients  to the nonzero elements $x_i$ of $M$. For any $0\neq x\in M$, we choose its coefficient $c_x$ to be  either $\frac{m}{(m,x)}$ or  $-\frac{m}{(m,x)}$, ensuring that $c_xx=0$ in $\mathbb{Z}_m$. If such an  assignment is possible, the permutational sum will be zero . 

First, observe that  $\frac{m}{(m,x)}$ and $-\frac{m}{(m,x)}$ are the same if and only if $(m,x)=1$.  Note that for each $i$, $p_i$ divides $(m,x_i)$ for all $x_i$, except two. Hence  there is no triple $x_1, x_2, x_3$ of the elements for which $(m,x_i)$ would be the same. Thus we can assign a different coefficient to each $x_i\neq 0$, except when there exist two of them, for which  $(m,x_i)=1$. But this is exactly the exceptional case $M=\{0, 0, \ldots, 0, c, -c \}$, where $(c,m)=1$.
\end{proof}

\bigskip


\section{Special cases of Theorem \ref{even}}

\medskip

\noindent In this section we prove that Theorem \ref{even} holds  for some specially structured multisets.

\begin{lemma} \label{hulye}
Let $m=2n$, $n>1$ odd, and let $M$ be a multiset in $\mathbb{Z}_m$ consisting of two blocks of size $n$ in the form $T_1= \{a, \ldots, a, a+b, a-b\}$ and $T_2= \{c, \ldots, c\} \pmod n$, where $(b,n)=1$. If one of the blocks contains elements from only one parity class then Theorem \ref{even} holds.
\end{lemma}

\begin{proof}
First we obtain a permutation for which $n$ divides the permutational sum of $M$. We choose an element $c^*$ from $T_2$. Assume that $c^*\neq a-b \pmod n$ and exchange $c^*$ with $a-b \in T_1$. (If the assumption does not hold then we pick $a+b$ instead of $a-b$ and continue the proof similarly.)
This way we get two blocks $T_1'$ and  $T_2'$, which do not have the exceptional structure mod $n$. Thus there exists a block preserving permutation ensuring that $n$ divides the obtained permutational sums of $T_1'$ and  $T_2'$, thus $n$ also divides the permutational sum of $M$.

 We assume that both odd and even elements occur in $M$, otherwise either $2$ is  trivially a  divisor of $\Phi$ or $M$ has exceptional homogeneous structure.
If the relation  $m\mid \Phi$ does not hold, then we apply the braid trick by looking at the pairs $(x_i, x_{i+n})$. If a pair consists of an odd and an even element, then we may transpose them and the proof is done.

 Otherwise the exchangement of $c^*$ and $a-b$ must have destroyed the property of having a uniform block mod $2$ among  $T_1$ and  $T_2$, that is, $c^*$ and $a-b$ have different parity. Since the choice of $c^* \in T_2$ was arbitrary, we may assume that $T_2= \{c, \ldots, c\} \pmod {2n}$. Moreover, since the braid trick did not help us, every element in $T_1$  congruent to $a$ mod $n$ must have the same parity as the elements $c$, and the parity of element $a+b$ must coincide with that of $a-b$.

In this remaining case consider the blocks $\{a, \ldots, a, c, c\}$ and \mbox{$\{c,\ldots,c,a+b,a-b\}$}. First, if $c\not \equiv a \pmod n$, then neither block  is exceptional as a multiset in  $\mathbb{Z}_n$, hence an appropriate block preserving permutation ensures that $n$ divides the permutational sum. If the permutational sum happens to be odd, then    a  suitable transposition via the braid trick will increase its value by $n$, for the first block contains elements from the same parity class in  contrast to the second.
Finally, if $c\equiv a \pmod n$, then either $c=a$ is even, providing that $M$ has inhomogeneous exceptional structure, or $c=a$ is odd, in which case the permutational sum will be zero if we set $a_{n}=a+b$ and $a_{2n}=a-b \pmod n$.

\end{proof}

\begin{lemma}\label{atlast} Let $m=2^kn > 4$, $n$ odd and $k>1$. Let $M$ be a multiset of $\mathbb{Z}_m$, consisting of two even elements and $m-2$ odd elements having residue $c$ mod ${2^{k-1}}$. Then the permutational sum of $M$ admits the value zero.
\end{lemma}

\begin{proof}
 Denote the even elements by $q_1$ and $q_2$. We distinguish the elements having residue $c$ mod ${2^{k-1}}$ according to their residues mod ${2^{k}}$, which  are  $c$ and  $c^* \equiv c+2^{k-1} \pmod {2^k}$.  We may suppose that the number of elements $c$ is greater than or equal to the number of elements $c^*$.

First we solve the case $n=1$ meaning $m=2^k$, $k>2$.
 Taking $a_{m/2}=q_1$, $a_{m}=q_2$, the permutational sum will be divisible by $m/2$. If there is no element $c^*$, then the permutational sum is in fact divisible by $m$.
If there exist some elements $c^*$ among the odd elements and the permutational sum is not yet divisible by $m$, then    a transposition between two elements $c$ and $c^*$ whose indices differ by an odd number will result in a zero permutational sum mod $m$.\\

Turning to the general case $n>1$, we initially order the elements as follows. Even elements precede the others,  elements $c$ mod ${2^k}$ precede the elements $c^*$ mod ${2^k}$, and equivalent elements mod $m$ are consecutive. Form $2^k$ blocks of equal size $n$. 

With an argument similar to the one in  Section 2.2 we arrive at two cases. Either we obtain a permutational sum congruent to zero mod $n$ after a block preserving permutation, or the structure of the blocks are as follows: there is exactly one exceptional block (as a multiset in  $\mathbb{Z}_n$) and  the other blocks only admit a zero permutational sum mod $n$ meaning that each of them consists of equivalent elements mod $n$.\\


 \noindent \textsl{Case 1)} Consider the block preserving permutation, which results in a permutational sum $\Phi_0$ divisible by $n$.
We modify this permutation, if necessary, to get one corresponding to a zero permutational sum mod ${2^k}$, while the divisibility by $n$ is preserved.

We denote by $f$ and $g$ the indices of $q_1$ and $q_2$ in the considered permutation.
Thus $$\Phi_0 \equiv c\frac {2^k(2^k-1)}{2}+  f(q_1-c) + g (q_2-c)\pmod {2^{k-1}}. \ \ \ \ \ \ \ \ \ \  (*)$$

\noindent Note that $\{ln: l=0,1, \ldots, 2^k-1\}$ is a complete system of residues mod ${2^k}$. Let $l$ be the solution of the congruence $$(q_1-c)ln\equiv -\Phi_0 \pmod {2^k}. $$
Thus transposing $q_1=a_f$ with $a_{f+ln} $ implies  that 

\[\Phi_1\equiv \left\{ \begin{array}{ll} 0 \pmod {2^k} & \textrm{if $a_{f+ln}\equiv c \pmod {{2^k}}$} \\ 2^{k-1}  \pmod {2^k} & \textrm{if  $a_{f+ln}\equiv c^* \pmod {{2^k}}$.} \end{array} \right. \] 

\noindent The relation $n\mid \Phi_1$ still holds.  So  in the case when $a_{f+ln}\equiv c \pmod {{2^k}}$ we are done, and if $a_{f+ln}\equiv c^* \pmod {{2^k}}$ we have to increase the value of the permutational sum by $2^{k-1}n$ mod $m$.
Recall that each element in the second block is $c$ mod ${2^k}$.
Therefore transposing $a_f\equiv c^*\pmod {{2^k}}$ with $a_{f+n}\equiv c \pmod {{2^k}}$ in this latter case does the job. \\

\noindent \textsl{Case 2) } One of the blocks (not necessarily the first one) has the exceptional structure, while every other is homogeneous  mod ${n}$. 
We can still argue as in the previous case if, performing the following operation, we can destroy the exceptional structure without changing the position of the even elements $q_1, q_2$ and the entire second block.
Namely, we try to transpose two nonequivalent elements mod $n$, one from the exceptional block and one from another block.
If this is not possible with the above mentioned constraints, then  the exceptional block must be among the first two.  
Furthermore, every element congruent to $c$ mod ${2^k}$ in the first two blocks must be equivalent mod $n$. 
Thus we only have to deal with the following structure: the first block is the exceptional one,  $q_1$ and $q_2$ correspond to $a+b$ and $a-b$ in the exceptional structure, all the other elements contained in the first two blocks are equivalent mod $m$ (and congruent to $c $ mod ${2^k}$), and the remaining blocks are all homogeneous mod $n$.


 
Exchanging $q_2$ with any element from the second block destroys the exceptional structure of the first block, which means that after a suitable block preserving permutation  the permutational sum of each block becomes $0$ mod $n$, ensuring $n\mid \Phi$ for the multiset. At this point the indices of the even elements are  $n$ and $2n$.

 Next, keeping the order inside the blocks we rearrange them so that the first and second blocks become the $2^{k-1}$th and $2^{k}$th, that is, $a_{{m}/{2}}=q_1$ and $a_m = q_2$. Hence, maintaining $n\mid \Phi$ we also achieve $2^{k-1}\mid \Phi$ via equality $(*)$.
 
Either we are done or  $\Phi \equiv 2^{k-1}n \pmod m$. The latter can only happen if there exists an element of type $c^*$. If a block contains both elements of type $c$ and $c^*$, then a transposition of a consecutive pair of them  within that block increases  $\Phi$ by $2^{k-1}n$. Otherwise there must exist a block containing only elements of type $c^*$. This implies the existence of a pair of $c$ and $c^*$ whose position differs by $n$. Their transposition increases  $\Phi$ by $2^{k-1}n^2\equiv 2^{k-1}n \pmod m$,   solving the case.
\end{proof}

 \bigskip


\section{ The case of even order}

\bigskip

\noindent One main difference between the odd and the even order case is due to the fact that  Lemma \ref{trafo1} (i) does not hold if $m$ is even, for $1+2+ \ldots + m$ is not divisible by $m$. That explains the emergence of the exceptional structure, see Proposition \ref{except}.
 

\begin{remark} \label{note}
It is easy to check that after a suitable permutation of the indices,  $\sum _iia_i\equiv m/2 \pmod m$ holds for the exceptionally structured multisets.
\end{remark}


In order to prove Theorem \ref{even}, we fix the notation $m=2^kn$, where $n$ is odd and $k>0$. Since the cases $m=2$ and $m=4$ can be checked easily, we assume that $m>4$ and prove the theorem by induction on $k$. \\

\noindent \textbf{Initial step}\\

\noindent We have $m=2n$, where  $n>1$ according to our assumption.
Take the multiset $M = \{ a_1,\dots ,a_m\}$ of $\mathbb{Z}_{m}$.  Arrange the elements in such a way that both the odd and the even elements are consecutive. 
Form two  consecutive blocks of equal size, denoted by $T_1$ and  $T_2$, each containing $n$ elements. Using the notation of Section $2$, the permutational sum of $M$ is 

$$\Phi= \sum_{j=1}^m ja_j =  \left(R_1+R_2 + \frac{m}{2}S_{2}\right) = R + nS_{2}.$$
Our first aim is to ensure that $n\mid \Phi$ holds after a well structured rearrangement of the elements. 

To this end, we may take an appropriate block preserving permutation providing that $n\mid R_i$  holds for $i=1,2$.
Such a permutation exists, except when at least one of the blocks are exceptional mod $n$. However it is enough to obtain a block preserving permutation for which $n\mid R$, and such a permutation exists via Lemma \ref{trafo1}   (iii), unless one of the blocks has exceptional structure (mod $n$) and the other consists of equivalent elements (mod $n$). This latter case was fully treated in Lemma \ref{hulye}.


 The next step is to modify the block preserving permutation such that $2 \mid \Phi$ also holds.\\
If it does not hold, then we try to transpose a pair $(a_i, a_{i+n})$ for which $a_i$ and $a_{i+n}$ have different parity, according to the braid trick. The permutational sum would change by $n$ (mod $m$) and we are done. If all pairs have the same parity, then all elements have the same parity. Therefore either $\Phi$ is automatically even or $M$ has
homogeneous exceptional structure. This completes the initial step. \\ 

\noindent\textbf{Inductional step}\\

\noindent Assume that $k>1$ and Theorem \ref{even} holds for every even proper divisor of $m$. 
Recalling Definition 2.4, we  choose a separable sequence  relative to the prime divisor $2$ of $m$ as an initial order.  Partition the multiset into two blocks of equal size, $T_1$ and $T_2$. 
Introduce $m^*:={m/2}=2^{k-1}n$, and assume first that $m^*\mid R_1+R_2$ can be achieved by a suitable block
preserving permutation.
 By induction, we can do it if both blocks as multisets have a structure different from the ones mentioned  in Proposition \ref{except}.
 If both blocks as multisets have exceptional structure mod ${m^*}$, then in view of Remark \ref{note} there exists a block preserving permutation for each block such that  $\sum _iia_i \equiv m/4  \pmod {m^*}$, thus $m^*\mid R_1+R_2$ holds.
 Finally, we can also achieve this relation if exactly one of the blocks has exceptional structure, and the permutational sum of the other block admits the value $m/4$ mod ${m^*}$.

Suppose that $m \mid R_1+R_2$ does not hold, otherwise we are done. 
Apply the braid trick and consider the pairs $(a_i, a_{i+2^{k-1}n})$. They must have the same parity, otherwise transposing them would make $\Phi$ divisible by $m$, which would complete the proof. Due to the separability of the initial order, all elements must have the same parity.

Consider now the pairs $(a_i, a_{i+2^{k-2}n})$. Either we can transpose the elements of such a pair to achieve a zero permutational sum, or the elements must have got the same residue mod ${2^2}$.
Apply this argument consecutively with exponent $s=1, 2, \ldots k$, for pairs $(a_i, a_{i+2^{k-s}n})$ and modulo $2^{k-s}$, respectively. Either $m \mid \Phi$ is obtained during this process  by a suitable transposition of a pair $(a_i, a_{i+2^{k-s}n})$ or all elements must have the same residue  $r$ mod ${2^k}$.

 If $r$ is odd, then $M$ has homogeneous exceptional structure described in Proposition \ref{except}. If $r$ is even, then $2^k$ would divide $\Phi$, for $\Phi \equiv r\frac {2^k(2^k-1)}{2} \pmod {2^k}$. Thus the concusion of the theorem holds in this case.\\

The remaining part of the proof is the case when only one of the blocks is exceptional mod $m^*$, and the permutational sum of the other block does not admit the value $m/4 \pmod {m^*}$. We refer to this latter condition by (**), and we may suppose that the second block is the exceptional one (otherwise we reverse the sequence). 
According to  Proposition \ref{except},
there are two cases to consider.\\


\noindent{\bf 5.1. The inhomogeneous case}\\

\noindent $T_2 = \{a,a, \ldots,a, q_1=a+b, q_2=a-b\}$ mod ${m^*}$, where $a$ is even and $(b,m)=1$. Note that $T_2$ contains both even and odd elements. Due to the separability of the initial order, all elements in $T_1$ have the same parity.

If $T_1$ consists of odd elements, then we exchange a pair of different odd elements mod ${m^*}$, one from each block. This way $T_2$  becomes non-exceptional. Moreover, an appropriate choice from $\{q_1, q_2\}$ ensures that $T_1$  does not become exceptional either. Thus $m^*$ will be a divisor of the permutational sum after a  suitable block preserving permutation. If $m\mid \Phi$ does not hold, we apply the braid trick for a pair $(a_i, a_{i+m^*})$ for which their parity differs and we are done.

If all elements of $T_1$ are even, then we try to transpose a pair of different even elements mod ${m^*}$, one from each block. Note that if it is possible, $T_1$ will not become exceptional. 
Hence after a block preserving permutation $m^*$ will be a divisor of the permutational sum. If $m\mid \Phi$ did not hold, we apply the braid trick for a pair $(a_i, a_{i+m^*})$ for which their parity differs and we are done.

Assume that  no appropriate transposition exists, that is, $T_1$ must consist of even elements having the same residue $a  $ mod ${m^*}$. It may occur that $M$ has the inhomogeneous exceptional structure.
Otherwise either $q_1+q_2=2a+m^*$, or there exists a pair $ a^{(1)} \not \equiv a^{(2)}\pmod{m}$ in $M$ such that   $a^{(1)} \equiv a^{(2)}\equiv a  \pmod {m^*}$.

We set the permutation now for these cases. Let  $q_1$ and $q_2$ be in the  positions  $1$ and $1+m^*$. Fix arbitrary positions for the rest of elements supposing that if a pair of type  $\{a^{(1)}, a^{(2)}\}$ exists, then the elements of such a  pair are  consecutive. Hence either we are done, or $\Phi \equiv m^* \pmod m$. In the latter case, note that  there must exist a pair of type  $\{a^{(1)}, a^{(2)}\}$ that is arranged consecutively.  Their transposition provides a zero permutational sum which completes the proof.\\

\noindent{\bf 5.2. The homogeneous case}\\

\noindent $T_2 = \{c,c,\ldots,c\}$ mod ${2^{k-1}}$ where $c$ is odd and {(**)} holds for $T_1$.\\

\noindent \textit{Subcase 1)}
     Every odd element $c'\in T_1$ is congruent to $c$ mod $2^{k-1}$.
\noindent Since $T_1$ is not exceptional mod $m^*$, it must contain some even elements. Thus $T_1$ consists of even elements and possibly also some odd elements having residue $c$  mod ${2^{k-1}}$.  
Choose an even element $q_1$ from $T_1$ and transpose it with $c$ in $T_2$. Since (**) holds for $T_1$, neither $T_1$ nor  $T_2$ become exceptional by this transposition. 

Take a permutation of each block for which the permutational sum is zero mod ${m^*}$. Either we are done or $\Phi \equiv m^* \pmod m$ holds. Look at the pairs $(a_i, a_i+m^*)$ according to the braid trick. If a pair takes different residues mod $2$, then their transposition makes the permutational sum divisible by $m$ and we are done. 
Otherwise we must have two even elements, and the others have residue $c$ mod ${2^{k-1}}$. Hence Lemma \ref{atlast} completes the proof.\\

\noindent \textit{Subcase 2)} There exists an odd $c' \in T_1$ for which $c'\not\equiv c \pmod {2^{k-1}}$.
 We transpose $c$ and $c'$ to obtain $T_2'= \{c',c,\ldots,c\}$ (mod $2^{k-1}$). 
We claim that $m^* \mid  \Phi$ holds for the new blocks $T_1'$ and $T_2'$ after  a suitable block preserving permutation. 

The permutational sum of $T_2'$ admits the value $m/4$ mod ${m^*}$. Indeed, it has a non-exceptional structure, hence it admits the value zero mod ${m^*}$, and then one transposition between $c'$ and another element is sufficient.  
Thus, neither (**) holds for $T_2'$ nor has it exceptional structure. 
Hence we may suppose that $m^* \mid  \Phi$ holds for the new blocks $T_1'$ and $T_2'$. Either we are done or $\Phi \equiv m^*$ (mod $m$). In the latter case  we need a transposition in $T_2'$ between $c'$ and another element congruent to $c$ mod ${2^{k-1}}$, for which the permutation sum changes by $m^*$ mod ${m}$. Such a transposition clearly exists.

\bigskip




\bigskip



\textbf{Acknowledgment.~} I am grateful to the anonymous referees for their valuable help in improving the presentation of the paper. 

\newpage

\bibliographystyle{plain}

\bigskip

\noindent Zolt\'an L\'or\'ant Nagy 

\noindent Department of Computer Science,

\noindent E\"otv\"os Lor\'and University, Budapest

\noindent H-1117, Budapest, P\'azm\'any P\'eter s\'et\'any 1/C

\noindent HUNGARY

\noindent {\tt e-mail: nagyzoltanlorant@gmail.com}

\end{document}